\documentclass[10pt]{amsart}
\newtheorem{thm}{Theorem}[section]

\newtheorem{lemma}[thm]{Lemma}
\newtheorem{cor}[thm]{Corollary}

\newtheorem{rmk}[thm]{Remark}

\newtheorem{notations}[thm]{Notations}
\newcommand{\coker}{{\rm coker}\,}

\newcommand{\N}{{\mathbb N}}

\newcommand{\tensor}{\otimes}
\newcommand{\sE}{{\mathcal E}}

\newcommand{\sL}{{\mathcal L}}

\newcommand{\sO}{{\mathcal O}}

\renewcommand{\tilde}{\widetilde}

\begin{document}
\title{Hilbert-Kunz functions of a  Hirzebruch surface}
\author{V. Trivedi}
\thanks{}
\address{School of Mathematics, Tata Institute of Fundamental Research,
Homi Bhabha Road, Mumbai-400005, India}
\email{vija@math.tifr.res.in}
\date{}
\begin{abstract}Here we compute Hilbert-Kunz functions of any nontrivial 
 ruled surface over ${\bf P}^1_k$, with respect to all ample line bundles 
on it.\end{abstract}
\maketitle

\section{Introduction}
Let $X$ be a projective variety over $ k = {\bar k}$, where $\mbox{char}~p > 0$. If $\sL$ is an ample 
line bundle on $X$, then one has the notions  of Hilbert-Kunz function 
$HK(X, \sL)$ and Hilbert-Kunz multiplicity 
$e_{HK}(X,\sL)$ as
$$HK(X,\sL) = HK_{R, R_+}~~~\mbox{and}~~~e_{HK}(X,\sL) = e_{HK}(R, R_+), $$
where $R = \oplus_{m\geq 0}H^0(X, \sL^m)$ and $R_+ = \oplus_{m > 0}H^0(X, \sL^m)$.
Recall that, for $q = p^n$,  one defines the {\em Hilbert-Kunz function} of $R$ with respect
to $I$ as
\[HK(R, I) = HK_{R, I}(q)=\ell(R/I^{[q]}),\]
where
$$I^{[q]}=\mbox{$n$-th Frobenius power of $I$}
=\mbox{ideal generated by $q$-th powers of elements of $I$}.$$
The associated {\em Hilbert-Kunz multiplicity} is defined to be
\[e_{HK}(R, I)=\lim_{n\to\infty} \frac{HK_{R,I}(q)}{q^{d}}.\]
where $d$ denotes the dimension of the ring $R$.
The Hilbert-Kunz function was introduced by Kunz \cite{K1},  
\cite{K2}, and studied extensively by Monsky. Monsky proved in
\cite{M1} that the limit defining the Hilbert-Kunz multiplicity exists.
For a survey on Hilbert-Kunz functions and Hilbert-Kunz multiplicity one can
 refer to a recent article by Huneke [H]. 

In this paper, we look at the following question:

{\it How do $e_{HK}(X, \sL)$ and $HK(X,\sL)$ behave as $\sL$ varies in the 
ample cone of line bundles on $X$?}

One knows the answer to this question for elliptic curves [NT] and nodal
 plane curves [M2]. 
Moreover for a line bundle $\sL = \sO((k-1)\rho)$ on a full flag variety $X$,
 one has computed in [NT], $e_{HK}(X, \sL^n)$ and 
$HK(X, \sL^n)$, for all  $n\geq 1$.

In this paper we  compute $e_{HK}(X, \sL)$ and $HK(X,\sL)$, for Hirzebruch surfaces $X = {\bf F}_a$
and every  ample line bundle $\sL$ on $X$. More precisely, in Theorem~\ref{t1}, for $\sL$
 given by $cD_1+dD_4$, where $D_1$ and $D_4$ are generators
 of $\mbox{Pic}(X)$,  we write $e_{HK}(X, \sL)$ and $HK(X, \sL)$
 explicitly in terms of $a$, $c$ and $d$ and periodic
 functions in $q$, determined  by $a, c$ and $d$. 
 
Here we crucially use the fact that $X$ is a nonsingular projective toric variety and a ruled surface. 

One motivation for the computation is to generate more 
complicated examples of the behaviour of Hilbert-Kunz functions and
 multiplicities.
In spite of extensive study of this subject, concrete examples or
 general principles are very few. One expects that such examples can be useful
 in future research, e.g., to make or test conjectures.

Note that if $R$ is an affine toric variety then  
Watanabe [W] proved that $e_{HK}(R)$ is a rational number and  W. Bruns [Br] 
asserted that $HK(R)(q)$ is a polynomial in $q$ 
with periodic coefficients. In case of a three dimensional affine toric variety, $e_{HK}$ has been given
 more explicitly in terms
 of the vectors of its semigroup ring, by Choi and An [CA].  Hilbert-Kunz functions
 for monomial ideals 
were determined by Conca [C].

\section{Lemmas}
Let $X = {\bf F}_a$ be the Hirzebruch surface with parameter $a\geq 1$, 
which is  a ruled  surface over ${\bf P}^1_k$, where $k$ is a field of characteristic $p > 0$. 
Then  $X$ is a projective toric variety given by a fan $\Delta $ in 
the lattice $N = {\bf Z}^2$, with cones $\sigma_1$, $\sigma_2$,
$\sigma_3$ and $\sigma_4$ given by  bases  
 $\{e_1, e_2\}$, $\{e_1, -e_2\}$, $\{-e_2, -e_1 + ae_2, \}$ and $\{-e_1+ae_2, e_2\}$ respectively (see Fulton [F], Section~1.1).

We recall (see [F], Chapter~3) that $\mbox{Pic}(X)$ is free on generators $D_1$
 and $D_4$, where $D_1$, $D_2$, $D_3$ and $D_4$
are the irreducible torus-invariant divisors  corresponding to the vectors
 $e_1$, $e_2$, $-e_1+ae_2$ and $-e_2$, respectively.
Since $X$ is a ruled surface given by 
$\pi: X = {\bf P}(\sO_{{\bf P}^1}(a)\oplus\sO_{{\bf P}^1}) \rightarrow {\bf P}^1_k$, we have
$$\sO_X(D_1) \simeq \pi^*{\sO}_{{\bf P}^1}(1), \mbox{and}~~\sO_X(D_4) \simeq 
{\sO}_{{\bf P}({\sE})}(1),$$
where $\sE = \sO_{{\bf P}^1}(a)\oplus\sO_{{\bf P}^1}$.
Moreover, by standard criteria (see [F], Chapter~3)
 a line bundle $\sO(cD_1+dD_4)$ is generated by global sections (is ample)
 if $c\geq 0$ and $d\geq 0$
 (if $c>0$ and $d>0$, respectively). 
Since any divisor 
$$D = b_1D_1+b_2D_2+b_3D_3+b_4D_4 \sim (b_1-ab_2+b_3)D_1 + (b_2+b_4) D_4$$
upto linear equivalence in $\mbox{Pic}(X)$,  the line bundle $\sO_X(D)$ is generated by global sections if $b_2+ b_4 
\geq 0$ and $b_1+b_3 \geq ab_2 $, and is ample if 
 $b_2+ b_4 > 0$ and $b_1+b_3 > ab_2 $.

If $M = Hom(N, {\bf Z})$ denotes the  dual lattice of $N$ then 
 $\chi = a_1e_1^*+a_2e_2^*$ in $M$ gives a 
 principal divisor $D_{\chi} = 
(-a_1)D_1 + (-a_2)D_2 + (a_1-aa_2)D_3 + (a_2)D_4$
 (following the convention of Fulton~[F]).
 Consider the following map
$$\frac{1}{q}:G^* = \frac{M}{qM}\longrightarrow \mbox{Pic}(X), $$
given by 
$\chi\mapsto \lfloor \frac{D_{\chi}}{q}\rfloor $.
We use the following result by Laso\'n-Michalek (Proposition~3.1 in [LM]) which was stated by Bondal~[B].

\vspace{5pt}

\noindent{\bf Theorem}\quad {\it Let $L = \sO_X(D)$ be a line bundle on a 
smooth projective toric 
variety $X$ then 
$$F_*({\sO}(D)) = \oplus_{\chi\in G^*}{\sO}\left({\lfloor \frac{D + D_{\chi}}{q}\rfloor}\right),$$
where $F:X\rightarrow X$ denotes the Frobenius map given by $x\mapsto x^q$}.

\vspace{5pt}

Now we fix an ample line bundle $\sL$ on $X$ given by the divisor $D = cD_1 + dD_4$, in particular, 
we have $c > 0$ and $d > 0$. 
Let $$R =  \oplus_{m\geq 0}H^0(X, \sL^m),  R_m = H^0(X, \sL^m)~~~ 
{\mbox {and  let}}~~ {\bf m} = \oplus_{m > 0}R_m.$$

Note that $R$ is a standard graded ring.
 We want to compute the Hilbert-Kunz function
of $R$ with respect to the maximal ideal ${\bf m}$, which is given by 

\begin{equation}\label{e2}HK(R, {\bf m})(q) = \ell(R_0) + \cdots + \ell(R_{q-1}) + \sum_{m\geq 0} 
\ell\left(\frac{R_{m+q}}{Im(R_1^{[q]}\tensor R_m)}\right).\end{equation}

But $R_{m+q}/Im(R_1^{[q]}\tensor R_m) = {\coker}(\Phi_m)$,
where $\Phi_m$ is given by the canonical map
$$\Phi_m: H^0(X, {\sO}(D))\otimes H^0(X, F_*{\sO}(mD))\longrightarrow H^0(X, F_*{\sO}((m+q)D))$$
Now, by the above result of [LM], 
$$F_*({\sO}(mD)) = \oplus_{\chi\in G^*}{\sO}\left({\lfloor \frac{mD + D_{\chi}}{q}\rfloor}\right),$$
where $\chi = a_1e_1^*+a_2e_2^*$,~  for $0\leq a_1, a_2 <q$, and 
$$D_{\chi} = -a_1D_1-a_2D_2+(a_1-aa_2)D_3+a_2D_4$$
 is the corresponding 
principal divisor.
Note that 
$${\lfloor \frac{mD + D_{\chi}}{q}\rfloor} = {\lfloor\frac{(mc-a_1)}{q}\rfloor}D_1 + {\lfloor\frac{-a_2}{q}\rfloor}D_2
+ {\lfloor\frac{(a_1-aa_2)}{q}\rfloor}D_3+ {\lfloor\frac{(md+a_2)}{q}\rfloor}D_4 .$$

{\it From here onwards, throughout this section we fix integers $m\geq 1$ , $q = p^n$
 and an ample line bundle $\sL$
on $X$ (in particular integers $c \geq 1$ and $d\geq 1$)}.
Let 
 $${\lfloor \frac{mD + D_{\chi}}{q}\rfloor} \sim \alpha(a_1, a_2)D_1 +
 \beta(a_1, a_2)D_4.$$

Note that 
\begin{equation}\label{e3}\begin{split}
\alpha(a_1, a_2) & =  {\lfloor \frac{mc -a_1}{q}\rfloor} + {\lfloor
 \frac{a_1 -aa_2}{q}\rfloor}
-a{\lfloor \frac{-a_2}{q}\rfloor},~~~\mbox{and}\\
\beta(a_1, a_2) & =  {\lfloor \frac{-a_2}{q}\rfloor} +
 {\lfloor \frac{md + a_2}{q}\rfloor}.\end{split}\end{equation}
Let
 $$A_{\alpha, \beta} = \{(a_1, a_2)\in Q\mid  \alpha(a_1, a_2) = \alpha,~ 
\beta(a_1, a_2) = \beta\} $$
where the bijective map 
$$Q = \{0, 1, \ldots, q-1\} \times  \{0, 1, \ldots, q-1\}\rightarrow
  \{\chi\in G^*\}$$
is given by $(a_1,a_2) \mapsto 
a_1e_1^*+a_2e_2^*$.
One can check that, if $A_{\alpha, \beta} \neq \phi$, 
then possible values of $\alpha$ and $\beta$ are in the range; 
$$-1\leq \beta(a_1, a_2) \leq \lfloor md/q\rfloor~~\mbox{ and}~~
t-1\leq \alpha(a_1, a_2) \leq t+ a,~~~\mbox{where}~~  t = \lfloor mc/q\rfloor.$$
Moreover, since $m, c, d, a, q$ are fixed, 
$\{A_{\alpha, \beta}\}_{\alpha, \beta}$  are disjoint sets.
 Let $|A_{\alpha, \beta}|$ denote the cardinality of 
the set $A_{\alpha, \beta}$. Then 
\begin{equation}\label{*}\displaystyle{F_*({\sO}(mD)) = \oplus_{\{t-1 
\leq \alpha \leq t+a, -1 \leq \beta \leq \lfloor md/q\rfloor\}} 
{\sO}(\alpha D_1 +\beta D_2)^{|A_{\alpha, \beta}|}.}\end{equation}

Let $\Phi_{\alpha, \beta}$ denote the canonical map
 (this is a restriction map of $\Phi_m$),
 $$\Phi_{\alpha, \beta}: H^0(X, {\sO}(cD_1+dD_4))\otimes H^0(X, {\sO}(\alpha D_1+ \beta D_4))
\longrightarrow H^0(X, {\sO}((c+\alpha)D_1+(d+\beta)D_4)).$$

For a map $\Phi_{-}$, let  $|\coker~\Phi_{-}|$ denote the $k$-vector space 
dimension of $\coker~\Phi_{-}$.

\begin{lemma}\label{l0} 
$$\begin{array}{lcl}
 |{\coker}~\Phi_{\alpha, \beta}| =  d\left(c+\alpha + 1 +\frac{(d-1)a}{2}\right)
   & {\rm if} & \alpha\geq -1, \beta = -1\\
 |{\coker}~\Phi_{\alpha, \beta}| =  0 & {\rm if} &  \alpha \geq 0, \beta \geq 0\\
 |{\coker}~\Phi_{\alpha, \beta}| =  (d+1)\left(c+\frac{ad}{2}\right)
& {\rm if} & \alpha = -1,  \beta = 0\\
 |{\coker}~\Phi_{\alpha, \beta}| =  c & {\rm if} & \alpha = -1, \beta \geq 1 
\end{array}$$\end{lemma}
\begin{proof} For the map
 $\pi:X = {\bf P}(\sE) = {\bf P}(\sO_{{\bf P}^1}\oplus \sO_{{\bf P}^1}(a))\rightarrow {\bf P}^1$, we have

$$\begin{array}{lcl}
h^0(X, \sO(\alpha D_1\oplus \beta D_4)) & = & h^0({\bf P}^1, 
\pi_*\sO(\alpha D_1\oplus \beta D_4)) = 
h^0({\bf P}^1, \pi_*(\pi^*\sO_{{\bf P}^1}(\alpha)
\tensor\sO_{{\bf P}(\sE)}(\beta)))\\
& =  & h^0({\bf P}^1, \sO_{{\bf P}^1}(\alpha)\tensor\pi_*\sO_{{\bf P}(\sE)}(\beta))
= h^0({\bf P}^1, \sO_{{\bf P}^1}(\alpha)\tensor S^{\beta}(\sE)).
\end{array} $$
Therefore $|{\coker}~\Phi_{\alpha, \beta}|$ is the dimension of the cokernel of the 
canonical map
$$H^0({\bf P}^1, \sO_{{\bf P}^1}(c)\tensor S^d(\sE))\tensor 
 H^0({\bf P}^1, \sO_{{\bf P}^1}(\alpha)\tensor
 S^{\beta}(\sE))
\longrightarrow H^0({\bf P}^1, \sO_{{\bf P}^1}(c+\alpha)\tensor S^{d+\beta}(\sE)).$$
Note that, for $\beta = -1$, $h^0({\bf P}^1, \sO_{{\bf P}^1}(\alpha)\tensor S^{\beta}(\sE)) = 0$.

Since $\sE = \sO_{{\bf P}^1}\oplus \sO_{{\bf P}^1}(a)$, we have,
 for $\beta \geq 0$,  
\begin{equation}\label{e1}\sO_{{\bf P}^1}(\alpha)\tensor
 S^{\beta}(\sE) = \sO_{{\bf P}^1}(\alpha)\oplus\sO_{{\bf P}^1}(\alpha +
 a)\oplus \cdots
\cdots \sO_{{\bf P}^1}(\alpha+\beta a). \end{equation}

Therefore the map $\Phi_{\alpha, \beta}$
has a cokernel of the following $k$-vector space dimensions: 
$$\begin{array}{lcl}
 |{\coker}~\Phi_{\alpha, \beta}| =   h^0(X, {\sO}((c+\alpha)D_1+(d+\beta)D_4))   & {\rm if} & \alpha\geq -1, \beta = -1\\
 |{\coker}~\Phi_{\alpha, \beta}| =  0 & {\rm if} &  \alpha \geq 0, \beta \geq 0\\
 |{\coker}~\Phi_{\alpha, \beta}| =   h^0(X, {\sO}((c+\alpha)D_1+(d+\beta)D_4)) & {\rm if} & \alpha = -1,  \beta = 0\\
 |{\coker}~\Phi_{\alpha, \beta}| =  c & {\rm if} & \alpha = -1, \beta \geq 1. 
\end{array}$$
Now from Equation~(\ref{e1}), we conclude the lemma.
\end{proof}

\begin{cor}\label{l1} Let $t =\lfloor mc/q\rfloor$ then 
$$\begin{array}{lcl}
|\coker~\Phi_m| & = & \displaystyle{
\sum_{\alpha\geq t-1}d\left(c+\alpha + 1 +\frac{(d-1)a}{2}\right)|A_{\alpha, -1}|
+ (d+1)\left(c+\frac{ad}{2}\right)\sum_{\beta\geq 0}|A_{-1, \beta}}|\\
& &
\displaystyle{- \left(dc+d(d+1)\frac{a}{2}\right) \sum_{\beta\geq 1}|A_{-1, 
\beta}|}.\end{array}$$
\end{cor}

\begin{proof} By Equation~(\ref{*}),
 $|{\coker}~\Phi_m| = \sum_{\alpha\geq t-1, \beta\geq -1}|A_{\alpha, \beta}|
|{\coker}~\Phi_{\alpha, \beta}|$.
Hence, by the above lemma,
$$|\coker~\Phi_m| = \sum_{\alpha\geq t-1}d\left(c+\alpha + 1 +\frac{(d-1)a}{2}\right)|A_{\alpha, -1}|
+ (d+1)\left(c+\frac{ad}{2}\right)|A_{-1, 0}|
+ \sum_{\beta\geq 1}c|A_{-1, \beta}|.$$
This  can be   rewritten as stated in the corollary.
\end{proof}

\vspace{5pt}

Note that from Equation~(\ref{e2}), we have
$$HK(X, \sL)(q) = \displaystyle{\sum_{m=0}^{q-1} h^0(X, \sL^m)} +
 \sum_{m\geq 0}|{\coker}~\Phi_m|.$$

The above proof of the lemma implies
that 
  $$\begin{array}{lcl} 
\displaystyle{h^0(X, \sL^m)} & = & \displaystyle{h^0({\bf P}^1, {\sO}(mc)\oplus {\sO}(mc+a)\oplus 
\cdots + {\sO}(mc+(md)a))}\\
 & = & \displaystyle{(md+1)\left[(mc+1)+\frac{amd}{2}\right]}\end{array}$$
which implies 
$$\begin{array}{lcl}
\displaystyle{\sum_{m=0}^{q-1} h^0(X, \sL^m)} & = 
& \displaystyle{\sum_{m=0}^{q-1}m^2d\left(c+\frac{ad}{2}\right) + m\left(d+c+\frac{ad}{2}\right) + 1}\end{array}$$
 
\begin{equation}\label{h1}\begin{split}
\displaystyle{\sum_{m=0}^{q-1} h^0(X, \sL^m)} & =  \displaystyle{q^3\left(\frac{d}{3}(c+\frac{ad}{2})\right) + q^2\left(-\frac{d}{2}(c+\frac{ad}{2})+
(d+c+\frac{ad}{2})
\frac{1}{2}\right)}\\
& \displaystyle{+ q\left(\frac{d}{6}(c+\frac{ad}{2})
 - (d+c+\frac{ad}{2})\frac{1}{2} + 1 \right)}\end{split}\end{equation}

Therefore we need to explicitly compute $|\coker~\Phi_m|$, for $m\geq 0$.

\vspace{5pt}

The following lemma determines the first term in the expression for 
$|\coker~\Phi_m|$ of Corollary~\ref{l1}. Instead of calculating each 
cardinality $|A_{\alpha, -1}|$, which amounts to counting points in two 
strips (depending on $\alpha$) in $Q$, we use coefficients of 
$|A_{\alpha, -1}|$ (as given in Corollary~\ref{l1}) to directly compute 
the
 sum. The idea is to first divide the expression in {\it two parts}; one part is a constant multiple of  the cardinality of 
 $\cup_{\alpha}A_{\alpha, -1}$, which  is the set of all integral points in a rectangle. To compute the  second part, 
we first redistribute $|A_{\alpha, -1}|'s$, using their coefficients. Then  
we split each $A_{\alpha, -1}$ as a disjoint union of $B_{\alpha}$ and 
${\tilde B_{\alpha}}$. Using this, the sum is reshuffled in a certain way,
 to carry out the computation.

Note that, by Equation~(\ref{e3}), for $(a_1, a_2) \in Q$, 
$$\beta(a_1, a_2) = -1 \iff a_2\neq 0~~~\mbox{and}~~md+a_2 < q.$$
  Therefore, for every $\alpha$
$$A_{\alpha, -1} = \{ (a_1, a_2)\in Q\mid 0 < a_2 < q-md,~~~ 
\alpha(a_1, a_2) = \alpha\}.$$

\begin{lemma}\label{l2}
If $m\geq q/d$ then  $A_{\alpha,-1} = \phi$, for every $\alpha $, therefore
$$\sum_{\alpha}d\left(c+\alpha+1+\frac{(d-1)a}{2}\right)|A_{\alpha, -1}| = 0.\quad\quad\quad\quad\quad\quad\quad\quad\quad\quad\quad\quad
\quad\quad\quad\quad $$
If  $0\leq m < q/d$ then, for $t = \lfloor mc/q\rfloor$,
$$\sum_{\alpha\geq t-1}d\left(c+\alpha+1+\frac{(d-1)a}{2}\right)|A_{\alpha, -1}|
= d(q-md-1)\left[ (q+m)\left(c+\frac{ad}{2}\right)+1\right].$$\end{lemma}

\begin{proof} Note $m\geq q/d$ $\Rightarrow q-md \leq 0 \Rightarrow A_{\alpha,
 -1} = \phi$. Therefore  the first assertion follows.      

Now we assume $0\leq m < q/d$.
Let $l_0 = c+(d-1)(a/2)$ and let us denote $A_{\alpha, -1} = A_{\alpha}$; then 
$$\begin{array}{lcl}
\displaystyle{\sum_{\alpha\geq t-1}d\left(l_0+\alpha+1\right)|A_{\alpha, -1}|}
& = & \displaystyle{d\Bigl[ (l_0+t)|A_{t-1}| + (l_0+t+1)|A_t| + \cdots + (l_0+t+a+1)|A_{t+a}|\Bigr] }\\
& = & \displaystyle{\left(d(l_0+t)\sum_{\alpha}|A_{\alpha}|\right) + d|A_t| + 2d|A_{t+1}| +
 \cdots + (a+1)d|A_{t+a}|}.\end{array} $$
Hence 
\begin{equation}\label{e4}
\displaystyle{\sum_{\alpha\geq t-1}d\left(l_0+\alpha+1\right)|A_{\alpha, 
-1}|} =  \displaystyle{\left(d(l_0+t)\sum_{\alpha}|A_{\alpha}|\right) + d\Bigl[|A_t| + 2|A_{t+1}| + \cdots + 
(a+1)|A_{t+a}|\Bigr]}\end{equation}

Since $$A_{t+l} =: A_{t+l, -1} =  \{ (a_1, a_2)\in Q\mid 0 < a_2 < q-md,~~~ 
\alpha(a_1, a_2) = t+l\},$$
we have $\bigcup_{\alpha}A_{\alpha} =  \{ (a_1, a_2)\in Q\mid 0 < a_2 < q-md,~~~ 
0\leq a_1 < q\}$.
Since this is a disjoint union of sets, we have 
$\sum_{\alpha}|A_{\alpha}| = q(q-md-1)$.  
                   
Now we express, each $A_{t+l}$ as a disjoint union of two sets,  
$$A_{t+l} = B_{t+l}\cup {\tilde B_{t+l}},$$
 where
$B_{t+l} = \{(a_1, a_2)\in Q \mid 0 < a_2 < q-md,~~ mc-tq < a_1,
~~\alpha(a_1, a_2) = t+l\}$ and 
 ${\tilde B_{t+l}} = \{(a_1, a_2)\in Q \mid 0 < a_2 < q-md,~~ mc-tq \geq a_1,
~~\alpha(a_1, a_2) = t+l\}$.

\vspace{5pt}

\noindent{\underline{Claim}}.~\quad 
$$B_{t+l} = \{(a_1, a_2)\in Q \mid 0 < a_2 < q-md,~~ mc-tq < a_1, 
(l+1)q\leq a_1+a(q-a_2) < (l+2)q\}$$
and 
$${\tilde B_{t+l}} = \{(a_1, a_2)\in Q\mid 0 < a_2 < q-md,~~ mc-tq \geq  a_1,~~
 lq\leq a_1+a(q-a_2) < (l+1)q\}\quad\quad\quad\quad\quad\quad $$
 
\noindent~{\underline {Proof of the claim}}.~\quad Note that $0 < a_2 $ implies that 
$\lfloor -a_2/q\rfloor = -1$, therefore $-a\lfloor -a_2/q\rfloor = a$.
 Also, $mc-tq = r$, 
for some integer $0\leq r < q$. Therefore 
$\lfloor (mc-a_1)/q\rfloor = t + \lfloor (r-a_1)/q\rfloor $.
If $(a_1, a_2)\in B_{t+l}$ then $-q < r-a_1 = mc-a_1-tq < 0$, which
 implies that 
$ \lfloor (r-a_1)/q\rfloor = -1$.
Hence 
$$\alpha(a_1, a_2) = t+l \iff \lfloor \frac{a_1-aa_2}{q}\rfloor
= l+1-a \iff (l+1)q\leq a_1+a(q-a_2) < (l+2)q.$$ 
Similarly one
 can check for ${\tilde B_{t+l}}$. This proves the claim.  

One can check that, $B_{t+a} = \phi $, hence $|B_{t+a}| = 0$. 
Therefore we can rewrite the following expression
$$\begin{array}{l}
\Bigl[|A_t| + 2|A_{t+1}| + \cdots + 
(a+1)|A_{t+a}|\Bigr]\\
 = \Bigl[|A_t| + |A_{t+1}| + \cdots + |A_{t+a}|\Bigr]
+ \Bigl[|A_{t+1}| + |A_{t+2}| + \cdots + |A_{t+a}|\Bigr]+ \cdots +
 \Bigl[|A_{t+a}|\Bigr]\\
= \Bigl[S_1+ |{\tilde B_{t}}|\Bigr] + \Bigl[S_2+ |{\tilde B_{t+1}}|\Bigr]
+\cdots + \Bigl[S_a + |{\tilde B_{t+a-1}}|\Bigr]+\Bigl[|{\tilde B_{t+a}}|\Bigr],
\end{array}$$
where 
$$ S_1 := \left[\left\{|B_t|+ |{\tilde B_{t+1}}|\right\}+ \cdots + 
\left\{|B_{t+a-1}|+ |{\tilde B_{t+a}}|\right\}\right]$$

$$S_2:=  \left[\left\{|B_{t+1}|+ |{\tilde B_{t+2}}|\right\} 
\cdots \left\{|B_{t+a-1}|+ |{\tilde B_{t+a}}|,\right\}\right]$$

$$\cdots S_a :=  \left[\left\{|B_{t+a-1}|+ |{\tilde B_{t+a}}|\right\}\right].
$$
Hence 
\begin{equation}\label{e5}
|A_t| + 2|A_{t+1}| + \cdots + 
(a+1)|A_{t+a}| = S_1 + S_2 +\cdots + S_a + |{\tilde B_t}|+
|{\tilde B_{t+1}}|+\cdots |{\tilde B_{t+a}}|.\end{equation}

Note that 
$$|{B_{t+l}}|+|{\tilde B_{t+l+1}}| = \displaystyle{|\{(a_1, a_2)\in Q
\mid 0 < a_2 < q-md,~~ (l+1)q\leq a_1+a(q-a_2) < (l+2)q\}|}.$$

Therefore 
$$\begin{array}{ll}
S_1 
& = \displaystyle{|\{(a_1, a_2)\in Q\mid 0 < a_2 < q-md,~~ 0\leq a_1 
< q,~~q\leq a_1+a(q-a_2)\}|\quad\quad\quad}\\
& = \displaystyle{|\{({\tilde a_1}, a_2)\in {\bf Z}^2\mid 0 < a_2 < q-md,~~ 
(a-1)q\leq {\tilde a_1} < aq,~~0\leq {\tilde a_1} - aa_2 \}|}\end{array}$$

and
$$\begin{array}{ll}
S_2 &
= \displaystyle{|\{(a_1, a_2)\in Q\mid 0 < a_2 < q-md,
~~0\leq a_1 < q,~~~2q\leq a_1+a(q-a_2)\}|}\\
& = \displaystyle{|\{({\tilde a_1}, a_2)\in {\bf Z}^2\mid 0 < a_2 < q-md,~~ 
(a-2)q\leq {\tilde a_1} < (a-1)q,~~0\leq {\tilde a_1}- aa_2 \}|}
\end{array}$$

$$\vdots$$ 
 
$$\begin{array}{ll}
S_a & 
= \displaystyle{|\{(a_1, a_2)\in Q\mid 0 < a_2 < q-md,~~~0\leq a_1 < q, ~~aq\leq a_1+a(q-a_2)\}|}\\
& = \displaystyle{|\{({\tilde a_1}, a_2)\in {\bf Z}^2\mid 0 < a_2 < q-md,~~ 
0\leq {\tilde a_1} < q,~~0\leq {\tilde a_1} - aa_2 \}|}
\end{array}$$
Hence 
$$S_1 + S_2 + \cdots + S_a = \displaystyle{|\{({\tilde a_1}, a_2)\in 
{\bf Z}^2\mid ~0 < a_2 < q-md,~~~0\leq {\tilde a_1}
 < aq,~~~\mbox{and}~~0\leq {\tilde a_1}-aa_2 \}|}.
$$
We note that, by hypothesis  $q > md$, and
 $0 < aa_2 \leq a(q-md)\leq aq$, provided  $0 < a_2 < q-md $. Hence 
we can write 
\begin{equation}\label{e6} 
S_1 + S_2 + \cdots + S_a  = T_1 + T_2,\end{equation}
 where 
$$\begin{array}{lcl}
T_1 & =  & 
\displaystyle{|\{({\tilde a_1}, a_2)\in {\bf Z}^2\mid ~0 < a_2 < q-md,~~
0\leq {\tilde a_1}
 < a(q-md),~~~0\leq {\tilde a_1}-aa_2 \}|}\\
& = &\displaystyle{|\{({\tilde a_1}, a_2)\in {\bf Z}^2\mid
~0 < a_2 < q-md,~~~aa_2 \leq {\tilde a_1} < a(q-md)\}|}\end{array}$$
and 
$$\begin{array}{lcl}
T_2 & = &  
\displaystyle{|\{({\tilde a_1}, a_2)\in {\bf Z}^2\mid ~0 < a_2 < q-md,~~~
a(q-md)\leq {\tilde a_1} < aq,~~~~0\leq {\tilde a_1}-aa_2\}|}\\
& = & \displaystyle{|\{({\tilde a_1}, a_2)\in 
{\bf Z}^2\mid ~0 < a_2 < q-md,~~~a(q-md)\leq {\tilde a_1} < aq,\}|}.
\end{array}$$
We note that 
$$T_1 =  \displaystyle{\sum_{a_2 = 1}^{q-md-1}a(q-md-a_2) = 
\frac{a}{2}(q-md)(q-md-1)}.$$
Next we observe that 
$$\begin{array}{l}
|{\tilde B_t}|+|{\tilde B_{t+1}}|+\cdots |{\tilde B_{t+a}}|\\
 = \displaystyle{|\{(a_1, a_2)\in Q\mid 0\leq a_1 \leq mc-tq,~~~ 
0 < a_2 < q-md,~~~0\leq a_1+a(q-a_2)<(a+1)q\}|}\\
 = \displaystyle{|\{(a_1, a_2)\in Q\mid ~0\leq a_1 \leq mc-tq,~~
 0 < a_2 < q-md~\}|}\\
= |\{({\tilde a_1}, a_2)\in {\bf Z}^2\mid aq\leq {\tilde a_1} \leq mc-tq + aq,~~
0 < a_2 < q-md\}|.\end{array}$$

Since $mc-tq \geq 0$ and $amd \geq 0$, we have that $aq-amd 
\leq aq \leq aq+mc-tq$, and so 
$$\begin{array}{l}
T_2 + |{\tilde B_t}|+|{\tilde B_{t+1}}|+\cdots |{\tilde B_{t+a}}|\\
 = \displaystyle{|\{({\tilde a_1}, a_2)\in {\bf Z}^2\mid 
\ a(q-md)\leq {\tilde a_1}\leq aq+mc-tq,~~0 < a_2 < q-md\}|}\\
 = (q-md-1)(mc-tq+amd+1).\end{array}$$
Hence, for $md < q$, by Equations~(\ref{e5}) and (\ref{e6}), we have   
$$d|A_t| + 2d|A_{t+1}| + \cdots + (a+1)d|A_{t+a}| = d\left[T_1+ \left(T_2 +
|{\tilde B_t}|+|{\tilde B_{t+1}}|+\cdots |{\tilde B_{t+a}}|\right)\right]$$

$$\begin{array}{l}
= \displaystyle{\frac{ad}{2}(q-md)(q-md-1) + d(q-md-1)(mc-tq+amd+1)}.
\end{array}$$

Therefore, by Equation~(\ref{e4}),  
$$\begin{array}{l}
 \sum_{\alpha\geq t-1}d\left(l_0+\alpha+1\right)|A_{\alpha, -1}|\\
\quad = \displaystyle{\sum_{\alpha\geq t-1}d(l_0+t)|A_{\alpha, -1}| +  d(q-md-1)\left[\frac{a}{2}(q-md) + (mc-tq+amd+1)\right]}\\
\quad = \displaystyle{d(l_0+t)q(q-md-1) +  d(q-md-1)\left[\frac{a}{2}(q-md) + (mc-tq+amd+1)\right]}\\
\quad = \displaystyle{dq(q-md-1)\left(c+\frac{(d-1)a}{2}\right) +  d(q-md-1)\left[\frac{a}{2}(q-md) +
 (mc+amd+1)\right]}\\
\quad = \displaystyle{d(q-md-1)\left[(q+m)\left(c+\frac{ad}{2}\right)+1\right]}.\quad\quad\quad\quad\quad\quad\quad\quad
\end{array}$$
Hence the proof.\end{proof}

\vspace{5pt}

Now we compute $\sum_{\beta \geq 0}|A_{-1, \beta}|$ and
$\sum_{\beta \geq 1}|A_{-1, \beta}|$.
Note that, from  Equation~(\ref{e2}), for $(a_1, a_2) \in Q$, 
$$\begin{array}{lcl}
\beta(a_1, a_2)\geq 0 & \iff &~~~\mbox{either}~~~
 (1)~~~a_2 = 0~~~\mbox{or}~~~(2)~~~a_2 \neq 0~~~\mbox{and}~~~q-md \leq a_2,\\
 \beta(a_1, a_2)\geq 1 & \iff & ~~~\mbox{either}~~~
 (1)~~~a_2 = 0, md \geq q~~~\mbox{or}~~~~(2)~~~~a_2 \neq 
0~~~\mbox{and}~~2q-md \leq a_2.\\
\alpha(a_1, a_2) = 
-1 & \iff &\mbox{either}~~~(1)~~~a_2 = 0, ~~~mc < a_1 
~~~\mbox{or}~~(2)~~a_2 \neq 0,~~~mc < a_1,~~~a_1+a(q-a_2) < q.\end{array}$$ 
Therefore 
(note that the condition $a_1+a(q-a_2) < q$ implies that $a_2 \neq 0$, 
as $a\geq 1$) 

$$\bigcup_{\beta \geq 0}A_{-1, \beta} = \{(a_1, 0)\in Q\mid mc 
< a_1\} \cup \{(a_1, a_2)\in Q\mid mc < a_1,~~~q-md \leq a_2,~~~a_1+
a(q-a_2) < q\},$$
and 
 $$\begin{array}{l}
\bigcup_{\beta \geq 1}A_{-1, \beta}\\
 = \{(a_1, 
0)\in Q\mid mc < a_1, ~~md \geq q\} \cup \{(a_1, a_2)\in Q\mid mc < a_1,~~~2q-md \leq 
a_2,~~~a_1+a(q-a_2) < q\}.\end{array}$$ 
Since $\{A_{-1, \beta}\}_{\beta}$ are disjoint sets, we have 
$|\bigcup_{\beta}A_{-1, \beta}| = \sum_{\beta}|A_{-1, \beta}|.$

\begin{lemma}\label{l3} 
$$\begin{array}{lll} \mbox{(1)}\quad \displaystyle{\sum_{\beta \geq 
0}|A_{-1, \beta}|} = & 
\displaystyle{(md+1)\left(q-mc-1-\frac{amd}{2}\right)},& 
~~\mbox{if}~~\displaystyle{0\leq m < \frac{q}{(c+ad)}},\\ 
\mbox{(2)}\quad \displaystyle{\sum_{\beta \geq 0}|A_{-1, \beta}|} = & 
\displaystyle{\lceil \frac{q-mc}{a}\rceil \left(q-mc-1-\frac{a}{2} 
\left(\lceil 
\frac{q-mc}{a}\rceil-1\right)\right)},&~~\mbox{if}~~\displaystyle{\frac{q}{(c+ad)}\leq 
m < \frac{q}{c}},\\ \mbox{(3)}\quad \displaystyle{\sum_{\beta \geq 
0}|A_{-1, \beta}|} = & 0, &~~~\mbox{if}~~\displaystyle{ \frac{q}{c}\leq 
m}. \end{array}$$\end{lemma}

\begin{proof} For $m > 0$, we can write
$$\bigcup_{\beta \geq 0} 
A_{-1, \beta} = \{(a_1,0) \in Q\mid mc< a_1\}~~\bigcup~
 \displaystyle{\bigcup_{i=1}^{md} X_{i}},$$
where $X_i = \{(a_1, q-i)\in Q \mid mc< a_1 < q-ai\}.$

\noindent(1)\quad Suppose 
$0\leq m < q/(c+ad)$. For $m =0$, 
$$\bigcup_{\beta \geq 0} 
A_{-1, \beta} = \{(a_1,0) \in Q\mid 0 < a_1\}  \implies 
\sum_{\beta \geq 0}|A_{-1, \beta}| = q-1,$$ which satisfies the formula 
$(1)$ in this case.

Let $m \geq 1$.
Now, for all $1\leq i \leq md$, we have  $mc < q-ai$, which implies 
$|X_i| = q-mc-ai-1$, for all $1\leq i\leq md$.

Therefore
$$\sum_{\beta \geq 0}|A_{-1, \beta}| = (q-mc-1) + \sum_{i=1}^{md}(q-mc-ai-1) =  (md+1)(q-mc-1-\frac{amd}{2}).$$

\noindent(2)\quad Suppose $q/c > m \geq q/(c+ad)$ then 
$mc  < q-ai \iff 1\leq i \leq {\lceil (q-mc)/a  \rceil-1}$.
Moreover $\lceil (q-mc)/a  \rceil-1 \leq md$,
therefore 
$$ \displaystyle{\bigcup_{i=1}^{md} X_{i}} = 
\displaystyle{\bigcup_{i=1}^{\lceil q-mc/a  \rceil-1} X_{i}}.$$
Hence 
$$\sum_{\beta \geq 0}|A_{-1,\beta}| = 
\sum_{i=0}^{\lceil(q-mc)/{a}\rceil-1}(q-ai-1-mc)
 = \displaystyle{\lceil \frac{q-mc}{a}\rceil \left(q-mc-1-\frac{a}{2}
\left(\lceil \frac{q-mc}{a}\rceil-1\right)\right)}.$$

\noindent(3)\quad Suppose $m\geq q/c$, then $A_{-1, \beta} = \phi$, for all
$\beta $. This implies $\sum_{\beta \geq 0}|A_{-1, \beta}|  = 0$.

This completes the proof of the lemma.\end{proof}

\begin{lemma}\label{l7} \mbox{(1)}\quad If $c \geq  d$ then 
$A_{-1, \beta} = \phi$, for all $\beta \geq 1$, therefore 
$$\sum_{\beta \geq 1} |A_{-1, \beta}| = 0, ~~~~\mbox{for all}~~~ m \geq 0.$$

\mbox{(2)}\quad If $c <  d$,  then 
$$ \begin{array}{lll}
\mbox{(i)}\quad\displaystyle{\sum_{\beta \geq 1} |A_{-1, \beta}|} = & 
0,& ~~\mbox{if}~~0\leq m < q/d\\ 
\mbox{(ii)}\quad\displaystyle{\sum_{\beta \geq 1} |A_{-1, \beta}|} = & \displaystyle{  
(md-q+1)\left(q-mc-1-\frac{a}{2}(md-q)\right)},&\mbox{if}~~\displaystyle{\frac{q}{d}\leq m < 
\frac{(a+1)q}{(c+ad)}}\\
\mbox{(iii)}\quad\displaystyle{\sum_{\beta \geq 1} |A_{-1, \beta}|} = & \displaystyle{ 
\lceil \frac{q-mc}{a}\rceil\left(q-mc-1-\frac{a}{2}
\left(\lceil \frac{q-mc}{a}\rceil -1\right)\right)},&\mbox{if}~~
\displaystyle{\frac{(a+1)q}{(c+ad)}\leq m <\frac{q}{c}}.\end{array}$$\end{lemma}
\begin{proof}
 If $md < q$ then $\beta = -1$ or $0$. Hence 
$\sum_{\beta \geq 1}|A_{-1, \beta}| = 0$, for $m < q/d$. 

If $c \geq d$ then  $m < q/c \implies m < q/d$ and therefore
$\sum_{\beta \geq 1}|A_{-1, \beta}| = 0$, for $m < q/c$. 
If $m\geq q/c$ then $A_{-1, \beta} = \phi $ for all $\beta $.
This proves  assertions (1) and (2)~(i) of the lemma.

\noindent~Now suppose  $c <  d$
and $ q/d \leq  m$. We then 
have
$$\bigcup_{\beta \geq 1} 
A_{-1, \beta} = \{(a_1,0) \in Q\mid mc< a_1, \}~~\bigcup~
 \displaystyle{\bigcup_{i=1}^{md-q} X_{i}},$$
where $X_i = \{(a_1, q-i)\in Q \mid mc< a_1 < q-ai\}.$

\noindent~(ii)\quad Let  $q/d\leq m < (a+1)q/(c+ad)$.
Then 
$$ mc < q-ai,~~~~~
\mbox{for all}~~~ 1\leq i\leq md-q.$$
Therefore $|X_i| = q-mc-1-ai$, for all $1\leq i \leq md-q$.
Hence 
$$\displaystyle{\sum_{\beta \geq 1} |A_{-1, \beta}|}  = 
\sum_{i=0}^{md-q}(q-mc-1-ai).$$ 

\noindent~(iii)\quad Let $(a+1)q/(c+ad) \leq m < q/c$. Then
$$ mc < q-ai~~~\iff 
~~~ i \leq \lceil \frac{q-mc}{a} \rceil-1,~~~\mbox{also}~~~
\lceil \frac{q-mc}{a} \rceil-1 \leq md - q.$$ 
Hence 
$$\displaystyle{\bigcup_{i=1}^{md-q} X_{i}} = 
\displaystyle{\bigcup_{i=1}^{\lceil {q-mc}/{a} \rceil-1} X_{i}}.$$
(This union could be empty, if $\lceil {q-mc}/{a} \rceil = 1$). Therefore
$$ \displaystyle{\sum_{\beta \geq 1} |A_{-1, \beta}|}  = 
 \sum_{i=0}^{\lceil {q-mc}/{a} \rceil-1}(q-mc-1-ai). $$ 
Now the proof of the lemma follows.\end{proof}

\vspace{5pt}

\section{Main Theorem}
Throughout this section, unless stated otherwise, we have $m\geq 0$ and 
$a, c, d\geq 1$ are integers and $q =p^n$, for some $n\in \N$, 
where $p$ is a prime number.
Let us fix the following notation:
\begin{notations}\label{n1}
$$\begin{array}{lcl}
\displaystyle{\epsilon_0} & = & \displaystyle{\lceil{q}/{(ad+c)}\rceil - {q}/{(ad+c)}}\\

\displaystyle{\bar{\epsilon_0}} & = & \displaystyle{\lceil{(a+1)q}/{(ad+c)}\rceil - {(a+1)q}/{(ad+c)}}\\

\displaystyle{\epsilon_1} & = & \displaystyle{\lceil{q}/{c}\rceil - {q}/{c}}\\

\displaystyle{\epsilon_2} & = & \displaystyle{\lceil{q}/{d}\rceil - {q}/{d}}\\

\displaystyle{\delta_0} &  = & \displaystyle{\lceil{q}/{a}\rceil - {q}/{a}}\\

a_1 &  =  & a/\mbox{g.c.d} (a,c)\\
\displaystyle{\Delta_0} &  = &
 \displaystyle{\left(\frac{q}{c}+{\epsilon_1-1}\right) - 
\lfloor\frac{q}{a_1c}+\frac{\epsilon_1-1}{a_1}\rfloor a_1  }\\

\displaystyle{\Delta_1} &   = & 
\displaystyle{\left(\frac{q}{(ad+c)} + \epsilon_0-1\right) - 
\lfloor \frac{q}{a_1(ad+c)}+\frac{\epsilon_0-1}{a_1}\rfloor a_1 }\\

\displaystyle{\Delta_2} &  = &
 \displaystyle{\left(\frac{(a+1)q}{(ad+c)}+{\bar{\epsilon}}_0-1\right) -
\lfloor \frac{(a+1)q}{a_1(ad+c)}+\frac{{\bar{\epsilon}}_0-1}{a_1}\rfloor a_1,}\\

M_1 & = & \displaystyle{\frac{aa_1\delta_0^2}{2}+\delta_0\frac{2a_1-a}{2}+
\frac{(a_1-1)(6a_1-aa_1-a)}{12a_1}}.
\end{array}$$

\end{notations}
\vspace{10pt}
Note that if $a =1$ then $\delta_0 = \Delta_0 = \Delta_1 = \Delta_2 = M_1 =
0$. If $a\mid c$ ({\it i.e.}, $a_1 = 1$) then 
$ \Delta_0 = \Delta_1 = \Delta_2 =0$.

\begin{thm}\label{t1} Let $ X = {\bf F}_a$ be a ruled surface over a field of 
characteristic $p > 0$ and let $\sL$ be an ample line bundle on $X$ given by 
$\sL = cD_1+dD_4$, where 
 $a, c, d$ are arbitrary positive integers.
Let $ q = p^n$, where $n\geq 1$.

If $c \geq  d$ then  $HK(X, \sL)(q) = (*)$, where 

$$\begin{array}{lcl}
(*) 
& = & \displaystyle{q^3\left((c+\frac{ad}{2})\left[\frac{d}{3}+
\frac{(d+1)d}{6c(ad+c)}+\frac{1}{2}+\frac{1}{6d}\right]\right)}\\

& & \displaystyle{+ q^2\left( (c+\frac{ad}{2}) (d+1)
 \left[\frac{1}{4c}+\frac{1}{4(ad+c)}-\frac{d}{2(ad+c)}-\frac{1}{2d}\right]+
\frac{d+1}{2}\right)}\\

& & +\displaystyle{ q\left((c+\frac{ad}{2})(d+1) \left[\frac{1}{2}-\frac{d}{12}
 -\frac{1}{c}
 - \left[\frac{1}{a_1}M_1-1\right]\frac{ad}{c(ad+c)}
-\frac{d\epsilon_2^2}{2}+
\frac{d\epsilon_2}{2}-\epsilon_2\right] \right)}\\

& & \displaystyle{-q^0\left( (c+\frac{ad}{2})(d+1)
\left\{ \frac{(ad+c)\epsilon_0}{2}\left[\frac{(\epsilon_0-1)}{2} + 
\frac{ad+c}{3a}(\epsilon_0^2-\frac{3}{2}\epsilon_0 +\frac{1}{2})\right]\right.\right.}\\

& & \displaystyle{ + \frac{c\epsilon_1}{2}\left[\frac{(\epsilon_1-1)}{2} 
- \frac{c}{3a}(\epsilon_1^2-\frac{3}{2}\epsilon_1 +\frac{1}{2})\right]}\\

& & \displaystyle{
+

(\frac{\epsilon_1-\epsilon_0+\Delta_1-\Delta_0}{a_1})M_1+
\frac{a}{2}(\sum_{i=0}^{\Delta_0}\delta^2_{i} -
\sum_{i=0}^{\Delta_1}\delta^2_{i})}
\\

&& \displaystyle{\left.
+(1-\frac{a}{2})(\sum_{i=0}^{\Delta_0}\delta_{i} -
\sum_{i=0}^{\Delta_1}\delta_{i})+ \epsilon_0 -

\frac{c\epsilon_1}{2a}(\epsilon_1-1) + 
\frac{c\epsilon_0}{2a}(\epsilon_0-1)
+ \frac{d\epsilon_0}{2}(\epsilon_0-1)

\right\}}\\

& & \displaystyle{\left.+(c+\frac{ad}{2})(d\epsilon_2)
\left[\frac{\epsilon_2^2d}{3}-\frac{d\epsilon_2}{2}+\frac{d}{6}+
\frac{\epsilon_2}{2}-\frac{1}{2}\right]+\left[\frac{d^2\epsilon_2^2}{2}-\frac{d^2\epsilon_2}{2}+
d\epsilon_2\right]  \right)},\end{array}$$

and if $ c <  d $ then  
$$\begin{array}{lcl}
HK(X, \sL)(q) & = &  \displaystyle{(*) + q^3 
  \left\{ d(c+\frac{(d+1)a}{2}) \left[\frac{(a+1)^3}{6a(ad+c)}
-\frac{1}{6ac}-\frac{a}{6d}-\frac{1}{2d}+\frac{c}{6d^2} \right] \right\}}\\
& &- \displaystyle{q^2\left\{d(c+\frac{(d+1)a}{2})\left[\frac{(a+2)(a+1)^2}{4a(c
+ad)} + \frac{a-2}{4ac} -
\frac{(a+2)}{4d} - \frac{1}{d} +
\frac{c}{2d^2}\right]\right\}}\\
& & \displaystyle{+ 
 q^1f_1(a, c, d, {\bar{\epsilon_0}}, \epsilon_2, \delta_0) + q^0f_0(a, c, d, \epsilon_1, 
\epsilon_2, {\bar{\epsilon_0}},  \Delta_0, \Delta_2,\delta_0, \cdots, 
\delta_{a_1-1})},\end{array}
$$
where $\epsilon_i, \delta_i's$ and $\Delta_i's$ are the above periodic 
functions in $q$, and 
$f_i(x_1, \ldots)$ denote  polynomials in $x_1, \ldots$.
 Moreover $f_i(x_1,\ldots)$ can be written down explicitly by 
putting together the computations given in the proof. 
\end{thm}

\begin{rmk}\label{r1}\quad One can check that, for any  triple
 $(d, q, \epsilon_2)$, where $d, q$ are positive integers and $\epsilon_2 =
\lceil q/d\rceil - q/d$, we have  

$$\begin{array}{lcl}
\displaystyle{d\sum_{m=0}^{\lceil {q/d}\rceil -1}} 1  & = & 
  q + d\epsilon_2\\
\displaystyle{d\sum_{m=0}^{\lceil {q/d}\rceil -1} md} & =  & 
\displaystyle{\frac{q^2}{2} + q\left(d\epsilon_2 - \frac{d}{2}\right)
+ q^0\left(\frac{d^2\epsilon_2^2}{2} - \frac{d^2\epsilon_2}{2}\right)}\\

\displaystyle{\sum_{m=0}^{\lceil {q/d}\rceil -1} (q-md)} &  = & \displaystyle{\frac{q^2}{2d} 
+ \frac{q}{2} + \frac{d\epsilon_2}{2}(1-\epsilon_2)}\\

\displaystyle{\sum_{m=0}^{\lceil {q/d}\rceil -1} d(q-md-1)} & = & \displaystyle{\frac{q^2}{2} 
+ q\left(\frac{d}{2}-1\right) -q^0\left(\frac{d^2\epsilon_2^2}{2}-\frac{d^2\epsilon_2}{2}+d\epsilon_2\right)}\\

\end{array}$$
$$
\begin{array}{lcl}
\displaystyle{\sum_{m=0}^{\lceil {q/d}\rceil -1} m^2d^3} & = & \displaystyle{\frac{q^3}{3} + q^2\left(d\epsilon_2-
\frac{d}{2}\right)
 +q\left(d^2\epsilon_2^2-d^2\epsilon_2+\frac{d^2}{6}\right)}\\

& &

+\displaystyle{q^0\left(\frac{d^3\epsilon_2^3}{3}-\frac{d^3\epsilon_2^2}{2}+\frac{d^3\epsilon_2}{6}\right)}\\

\displaystyle{\sum_{m=0}^{\lceil {q/d}\rceil -1} d^2m(q-md-1)} & = & \displaystyle{\frac{q^3}{6} 
- q^2\left(\frac{1}{2}\right) - q\left(\frac{d^2\epsilon_2(\epsilon_2-1)}{2}+\frac{d^2}{6} + d\epsilon_2 -
 \frac{d}{2}\right)}\\ 

 & & \displaystyle{-q^0\left((d^2\epsilon_2)(\frac{\epsilon_2^2d}{3}-\frac{d\epsilon_2}{2}+
\frac{d}{6}+\frac{\epsilon_2-1}{2})\right)}\\

\displaystyle{\sum_{m=0}^{\lceil {q/d}\rceil -1} (q-md)^2} & 
= & \displaystyle{\frac{q^3}{3d} + \frac{q^2}{2} 
+ q\left(\frac{d}{6}\right) + q^0\left(\frac{d^2\epsilon_2}{3}(\epsilon_2^2-\frac{3}{2}\epsilon_2 +
\frac{1}{2})\right)  }\end{array}$$

Note that $q$ here need not be a power of prime number. In particular
 the computation holds for triples like $(c+ad, (a+1)q, {\bar{\epsilon_0}})$.
\end{rmk}

\vspace{5pt}

\begin{lemma}\label{l8} Let $mc < q$ and 
let $\displaystyle{\delta_m}  =  \displaystyle{\lceil{(q-mc)}/{a}\rceil
 - {(q-mc)}/{a}}$. Then
\begin{enumerate}
\item $$
\sum_{m = \lceil {q/(c+ad)}\rceil}^{\lceil 
\frac{q}{c}\rceil -1}(\frac{a\delta_m^2}{2} +(1-\frac{a}{2})\delta_m-1) =  
q\left[\frac{ad}{c(ad+c)}(\frac{M_1}{a_1}
- 1)\right] + \left[ 
(\frac{\epsilon_1 - \epsilon_0 + \Delta_1- \Delta_0}{a_1})M_1 \right.$$
$$\quad\quad\quad\quad\quad\quad\quad\quad
 +\left.\frac{a}{2}(\sum_{i=0}^{\Delta_0}\delta^2_{i} -
\sum_{i=0}^{\Delta_1}\delta^2_{i})
 \displaystyle{+(1-\frac{a}{2})(\sum_{i=0}^{\Delta_0}\delta_{i} -
\sum_{i=0}^{\Delta_1}\delta_{i})
+ (\epsilon_0-\epsilon_1)}\right],$$
and 

\item $$\displaystyle{\sum_{m = \lceil {(a+1)q/ad+c}\rceil }^{\lceil 
q/c\rceil -1} - \frac{a}{2}\delta_m^2+ (\frac{a}{2}-1)\delta_m
= q\left[\frac{a(c-d)}{a_1c(ad+c)}(M_1)\right] + \left[
(\frac{{\bar{\epsilon_0}} - {\epsilon_1}+\Delta_0-\Delta_2)}{a_1})(M_1)
 \right.}$$
$$\quad\quad\quad\quad\quad\quad\quad +\displaystyle{
\frac{a-2}{2}(\sum_{i=0}^{\Delta_0}\delta_{i} -
\sum_{i=0}^{\Delta_2}\delta_{i})}
 \displaystyle{\left.
-\frac{a}{2}(\sum_{i=0}^{\Delta_0}\delta^2_{i} -
\sum_{i=0}^{\Delta_2}\delta^2_{i})\right]},
$$\end{enumerate}
where $\Delta_0, \Delta_1, \Delta_2$ and $M_1$
 are defined as in Notation~\ref{n1}.
\end{lemma}
\begin{proof}Note that 
$\delta_m$ is a periodic function in $m$ of period $a_1
 = a/\mbox{g.c.d.}(a,c)$, therefore $\delta_{a_1+x} = \delta_x$, for 
every integer $x \geq 0$.
Moreover 
$$\sum_{i=0}^{a_1-1}\delta_i = \sum_{i=0}^{a_1-1}(\delta_0+\frac{i}{a_1}) = 
\delta_0a_1 + \frac{a_1-1}{2}.$$
 
\noindent{Case~$(1)$}\quad
$$\begin{array}{lcl}
\displaystyle{\sum_{m=0}^{\lceil {q/c}\rceil -1}\delta_m} & = & (\delta_0+\cdots+\delta_{a_1-1})+\cdots
+(\delta_{(l-1)a_1}+\cdots +\delta_{la_1-1})+(\delta_{la_1}+\cdots + 
\delta_{\lceil {q/c}\rceil -1})\\
 & = & \displaystyle{l(\delta_0a_1+
\frac{a_1-1}{2})+(\delta_0+\cdots + \delta_{\Delta_0})},\end{array}$$
where 
$l = \lfloor(\lceil{q}/{c}\rceil-1)\frac{1}{a_1}\rfloor $ and 
 \begin{equation}\label{e9}
 \displaystyle{\Delta_0} = \left(\lceil\frac{q}{c}\rceil-1\right) - la_1 =
 \displaystyle{\frac{q}{c}+{\epsilon_1-1} - 
\lfloor\frac{q}{a_1c}+\frac{\epsilon_1-1}{a_1}\rfloor a_1}.\end{equation}

Similarly  
$$\sum_{m=0}^{\lceil {q/(ad+c)}\rceil -1}\delta_m =l_1(\delta_0a_1+
\frac{a_1-1}{2})+(\delta_0+\cdots + \delta_{\Delta_1}),$$
where 
$l_1= \lfloor(\lceil{q}/{(ad+c)}\rceil-1)\frac{1}{a_1}\rfloor $ and 

\begin{equation}\label{e10}\displaystyle{\Delta_1}  = 
 \displaystyle{\left(\lceil
\frac{q}{ad+c}\rceil-1\right) -l_1a_1} =
\displaystyle{\frac{q}{(ad+c)} + \epsilon_0-1 - 
\lfloor \frac{q}{a_1(ad+c)}+\frac{\epsilon_0-1}{a_1}\rfloor a_1. }
\end{equation}
Therefore 

$$\sum_{m={\lceil {q/(ad+c)}\rceil}}^{\lceil {q/c}\rceil -1}\delta_m =
(l-l_1)(\delta_0a_1+\frac{a_1-1}{2})+(
\sum_{i=0}^{\Delta_0}\delta_{i} -
\sum_{i=0}^{\Delta_1}\delta_{i})$$
and 
since 
$$\sum_{m=0}^{a_1-1}\delta_m^2 = a_1\delta_0^2 + 
(a_1-1)\delta_0 + \frac{(a_1-1)(2a_1-1)}{6a_1},$$

we have
$$\sum_{m={\lceil {q/(ad+c)}\rceil}}^{\lceil {q/c}\rceil -1}\delta_m^2 =
(l-l_1)\left[a_1\delta_0^2+ \delta_0(a_1-1)+\frac{(a_1-1)(2a_1-1)}{6a_1}\right]
+(\sum_{i=0}^{\Delta_0}\delta^2_{i} -
\sum_{i=0}^{\Delta_1}\delta^2_{i}),$$
where 
$$l-l_1= \frac{ad}{a_1c(ad+c)}q+(\frac{\epsilon_1-\epsilon_0+
\Delta_1-\Delta_0}{a_1}),$$ 
from Equations~(\ref{e9}) and (\ref{e10}).Let us write this as $l-l_1 = xq+y$.
Therefore
$$\sum\frac{a}{2}\delta_m^2+ \sum (1-\frac{a}{2})\delta_m-\sum 1$$
$$ = \frac{a}{2}(xq+y)\left[a_1\delta_0^2+\delta_0(a_1-1)+
\frac{(a_1-1)(2a_1-1)}{6a_1}\right] + 
\frac{a}{2}(\sum_{i=0}^{\Delta_0}\delta^2_{i} -
\sum_{i=0}^{\Delta_1}\delta^2_{i})$$

$$+(1-\frac{a}{2})(xq+y)\left[\delta_0a_1+\frac{a_1-1}{2}\right]
+ (1-\frac{a}{2})\left(\sum_{i=0}^{\Delta_0}\delta_{i} -
\sum_{i=0}^{\Delta_1}\delta_{i}\right)
- \lceil{q/c}\rceil+\lceil{q/(ad+c)}\rceil.$$
$$ =q[xM_1] +  yM_1 +  \frac{a}{2}(\sum_{i=0}^{\Delta_0}\delta^2_{i} -
\sum_{i=0}^{\Delta_1}\delta^2_{i}) + 
(1-\frac{a}{2})\left(\sum_{i=0}^{\Delta_0}\delta_{i} -
\sum_{i=0}^{\Delta_1}\delta_{i}\right)
- q\frac{ad}{c(ad+c)} + (\epsilon_0-\epsilon_1),$$
where 
$$M_1 = \frac{a}{2} \left[a_1\delta_0^2+\delta_0(a_1-1)+
\frac{(a_1-1)(2a_1-1)}{6a_1})\right] + 
(1-\frac{a}{2})\left[\delta_0a_1+\frac{a_1-1}{2}\right]$$
$$  =  \displaystyle{\frac{aa_1\delta_0^2}{2}+\delta_0\frac{2a_1-a}{2}+
\frac{(a_1-1)(6a_1-aa_1-a)}{12a_1}}.$$

Therefore 

$$\sum_{m={\lceil {q/(ad+c)}\rceil}}^{\lceil {q/c}\rceil -1}
 (\frac{a\delta_m}{2}+1)(\delta_m-1) = 
q\frac{ad}{ad+c}(\frac{M_1}{a_1}-\frac{1}{c}) +
\left[ 
(\frac{\epsilon_1 - \epsilon_0 + \Delta_1- \Delta_0}{a_1})M_1 \right.$$
$$\quad\quad\quad\quad\quad\quad +
\left.\frac{a}{2}(\sum_{i=0}^{\Delta_0}\delta^2_{i} -
\sum_{i=0}^{\Delta_1}\delta^2_{i})
 \displaystyle{+(1-\frac{a}{2})(\sum_{i=0}^{\Delta_0}\delta_{i} -
\sum_{i=0}^{\Delta_1}\delta_{i})
+ (\epsilon_0-\epsilon_1)}\right].$$

This proves part~$(1)$ of the lemma

\noindent~$(2)$\quad
As in the previous case, one can see that 
$$\sum_{m=\lceil {(a+1)q/(ad+c)}\rceil}^{\lceil q/c\rceil -1}\delta_m =
(l-l_2)(\delta_0a_1+
\frac{a_1-1}{2})+(\sum_{i=0}^{\Delta_0}\delta_{i} -
\sum_{i=0}^{\Delta_2}\delta_{i})$$

and
$$\sum_{m={\lceil {(a+1)q/(ad+c)}\rceil}}^{\lceil {q/c}\rceil -1}\delta_m^2 =
(l-l_2)\left[a_1\delta_0^2+ \delta_0(a_1-1)+\frac{(a_1-1)(2a_1-1)}{6a_1}\right]
+(\sum_{i=0}^{\Delta_0}\delta^2_{i} -
\sum_{i=0}^{\Delta_2}\delta^2_{i}),$$
where 
$l_2= \lfloor(\lceil{(a+1)q}/{(ad+c)}\rceil-1)\frac{1}{a_1}\rfloor $
and 
$$\displaystyle{\Delta_2} = (\lceil\frac{(a+1)q}{(ad+c)}\rceil-1)-l_2a_1 = 
\displaystyle{\frac{(a+1)q}{(ad+c)}+{\bar{\epsilon}}_0-1 -
\lfloor \frac{(a+1)q}{a_1(ad+c)}+\frac{{\bar{\epsilon}}_0-1}{a_1}\rfloor 
a_1}.$$

Therefore $l-l_2= \frac{a(d-c)}{a_1c(ad+c)}q+
(\frac{\epsilon_1-{\bar{\epsilon_0}}+
\Delta_2-\Delta_0}{a_1})$.

Hence 

$$\displaystyle{ 
-\sum \frac{a}{2}\delta_m^2 + \sum (\frac{a}{2}-1)\delta_m } $$
$$ = q\frac{a(c-d)}{a_1c(ad+c)}(M_1) + 
(\frac{{\bar{\epsilon_0} - \epsilon_1}+\Delta_0-\Delta_2}{a_1})(M_1)
+\displaystyle{
\frac{a-2}{2}(\sum_{i=0}^{\Delta_0}\delta_{i} -
\sum_{i=0}^{\Delta_2}\delta_{i})}
 \displaystyle{
-\frac{a}{2}(\sum_{i=0}^{\Delta_0}\delta^2_{i} -
\sum_{i=0}^{\Delta_2}\delta^2_{i})}.$$

This proves the part~$(2)$ and hence the lemma.
\end{proof}

\vspace{5pt}

\begin{lemma}\label{l6}Let 
$$\begin{array}{lcl}
A & = &\sum_{m=0}^{\lceil q/d\rceil-1}d(q-md-1)\left[ (q+m)\left(c+\frac{ad}{2}\right)+1\right]\\
& & + 
\displaystyle{(d+1)(c+\frac{ad}{2})
\left[\sum_{m=0}^{\lceil q/(c+ad)\rceil-1}(md+1)
\left(q-mc-1-\frac{amd}{2}\right)\right.}\\
& & +\displaystyle{\left.\sum_{\lceil q/(c+ad)\rceil}^{\lceil q/c\rceil -1}  
 \lceil \frac{q-mc}{a}\rceil \left(q-mc-1-\frac{a}{2}
\left(\lceil \frac{q-mc}{a}\rceil-1\right)\right)\right]}, \end{array}$$
then
$$\begin{array}{lcl}
A 
& = & \displaystyle{q^3\left((c+\frac{ad}{2})\left[
\frac{(d+1)d}{6c(ad+c)}+\frac{1}{2}+\frac{1}{6d}\right]\right)}\\

& & \displaystyle{+ q^2\left( (c+\frac{ad}{2}) (d+1)
 \left[\frac{1}{4c}+\frac{1}{4(ad+c)}-\frac{d}{2(ad+c)}\right]+
\left(\frac{d}{2}-1-\frac{1}{2d}\right)\left(c+\frac{ad}{2}\right)
+\frac{1}{2}\right)}\\

& & +\displaystyle{ q\left((c+\frac{ad}{2})(d+1) \left[\frac{1}{2}-\frac{d}{12}
 -\frac{1}{c}
 - \left[\frac{1}{a_1}M_1-1\right]\frac{ad}{c(ad+c)}
-\frac{d\epsilon_2^2}{2}+
\frac{d\epsilon_2}{2}-\epsilon_2\right] \right.}\\
& &
 +\displaystyle{\left.(c+\frac{ad}{2})(\frac{1}{2}-\frac{d}{6})+
(\frac{d}{2}-1)\right)}\\
& & \displaystyle{-q^0\left( (c+\frac{ad}{2})(d+1)
\left\{ \frac{(ad+c)\epsilon_0}{2}\left[\frac{(\epsilon_0-1)}{2} + 
\frac{ad+c}{3a}(\epsilon_0^2-\frac{3}{2}\epsilon_0 +\frac{1}{2})\right]\right.\right.}\\

& & \displaystyle{ + \frac{c\epsilon_1}{2}\left[\frac{(\epsilon_1-1)}{2} 
- \frac{c}{3a}(\epsilon_1^2-\frac{3}{2}\epsilon_1 +\frac{1}{2})\right]
+
(\frac{\epsilon_1-\epsilon_0+\Delta_1-\Delta_0}{a_1})M_1+
\frac{a}{2}(\sum_{i=0}^{\Delta_0}\delta^2_{i} -
\sum_{i=0}^{\Delta_1}\delta^2_{i})}
\\

&& \displaystyle{\left.
+(1-\frac{a}{2})(\sum_{i=0}^{\Delta_0}\delta_{i} -
\sum_{i=0}^{\Delta_1}\delta_{i})+ \epsilon_0 -

\frac{c\epsilon_1}{2a}(\epsilon_1-1) + 
\frac{c\epsilon_0}{2a}(\epsilon_0-1)
+ \frac{d\epsilon_0}{2}(\epsilon_0-1)

\right\}}\\

& & \displaystyle{\left.+(c+\frac{ad}{2})(d\epsilon_2)
\left[\frac{\epsilon_2^2d}{3}-\frac{d\epsilon_2}{2}+\frac{d}{6}+
\frac{\epsilon_2}{2}-\frac{1}{2}\right]+\left[\frac{d^2\epsilon_2^2}{2}-\frac{d^2\epsilon_2}{2}+
d\epsilon_2\right]  \right)},\end{array}$$
\end{lemma}
\begin{proof}

Note that $$(md+1)
\left(q-mc-1-\frac{amd}{2}\right) = 
md(q-mc-\frac{amd}{2})+ (q-mc-\frac{amd}{2}) -md -1$$ 
$$= \frac{2}{a}(\frac{q-mc}{2} - \frac{q-mc-amd}{2})
(\frac{q-mc}{2} + \frac{q-mc-amd}{2}) + (\frac{q-mc}{2} + \frac{q-mc-amd}{2})
-md -1$$
$$ = \frac{(q-mc)^2}{2a}+\frac{q-mc}{2} - 1 + 
\frac{(q-mc-amd)}{2}-\frac{(q-mc-amd)^2}{2a}-md.$$
On the other hand, for 
$$\displaystyle{\delta_m}  =  \displaystyle{\lceil{(q-mc)}/{a}\rceil
 - {(q-mc)}/{a}},$$
$$\lceil\frac{q-mc}{a}\rceil \left(q-mc-1-\frac{a}{2}
\left(\lceil\frac{q-mc}{a}\rceil-1\right)\right) 
 = (\frac{q-mc}{a}+\delta_m)\left(\frac{q-mc}{2}-1-
\frac{a\delta_m}{2}+\frac{a}{2}\right)$$
$$= \frac{(q-mc)^2}{2a}-\frac{(q-mc)}{a} + \frac{q-mc}{2} +
 (\frac{a}{2}-1)\delta_m-\frac{a}{2}\delta_m^2.$$

Now combining all the terms
we get
$$\begin{array}{lcl} 
 A  & = & 
\displaystyle{\sum_{m = 0}^{\lceil q/d\rceil -1}\left\{ d(q-md-1)(q+m)(c+\frac{ad}{2}) + d(q-md-1) \right\}}\\

& & \displaystyle{+(d+1)(c+\frac{ad}{2})\left[ \sum_{m = 0}^{\lceil q/c\rceil -1}
\left\{\frac{q-mc}{2} - 1 + \frac{(q-mc)^2}{2a}\right\}\right.}\\

& & \displaystyle{ + \sum_{m = 0}^{\lceil {q/c+ad}\rceil -1}
\left\{ \frac{(q-amd-mc)}{2} - \frac{(q-amd-mc)^2}{2a} \right\}}\\
& & \displaystyle{\left.-
\sum_{m = \lceil {q/c+ad}\rceil}^{\lceil \frac{q}{c}\rceil -1}
\left\{\frac{(q-mc)}{a} + (\frac{a\delta_m^2}{2}+(1-\frac{a}{2})\delta_m-1)
  \right\} 
- \sum_{m = 0}^{\lceil {q/c+ad}\rceil -1}md

\right]}.  
\end{array}$$

By Remark~\ref{r1},
we can check that different terms in $A$ can be written down as follows:
  $$\sum_{m = 0}^{\lceil q/d\rceil -1} d(q-md-1)(q+m) = q^3\left(\frac{1}{2}+\frac{1}{6d}\right)
+q^2\left(\frac{d}{2}-1-\frac{1}{2d}\right)$$
$$+q\left((d+1)\left[\frac{-d\epsilon_2^2}{2}+\frac{d\epsilon_2}{2}
-\epsilon_2\right]+\frac{1}{2}-\frac{d}{6}\right) -
 q^0\left((d\epsilon_2)\left[\frac{d\epsilon_2^2}{3}-\frac{d\epsilon_2}{2}+\frac{d}{6}+
\frac{\epsilon_2}{2}-\frac{1}{2}\right]\right)$$
 
and
$$ II(a) := \displaystyle{ \sum_{m = 0}^{\lceil \frac{q}{c}\rceil -1}
\left\{\frac{q-mc}{2} - 1 + \frac{(q-mc)^2}{2a}\right\}
  + \sum_{m = 0}^{\lceil \frac{q}{(c+ad)}\rceil -1}
\left\{ \frac{(q-amd-mc)}{2} - \frac{(q-amd-mc)^2}{2a} \right\}}$$

$$ = q^3\left(\frac{d}{6c(ad+c)}\right) + q^2\left(\frac{1}{4c}+\frac{1}{4(ad+c)}\right) + 
q\left(\frac{1}{2}-\frac{d}{12}-\frac{1}{c}\right)$$
$$-q^0\left((ad+c)\frac{\epsilon_0}{2}\left[\frac{(\epsilon_0-1)}{2} + 
\frac{ad+c}{3a}(\epsilon_0^2-\frac{3}{2}\epsilon_0 +\frac{1}{2})\right] + 
\frac{c\epsilon_1}{2}\left[
\frac{(\epsilon_1-1)}{2} -
\frac{c}{3a}(\epsilon_1^2-\frac{3}{2}\epsilon_1 +\frac{1}{2})\right]+\epsilon_1\right).$$

Now consider 
$$II(b)  :=  \displaystyle{
\sum_{m = \lceil {q/c+ad}\rceil}^{\lceil {q/c}\rceil -1}
\left\{\frac{(q-mc)}{a} + (\frac{a\delta_m^2}{2}+(1-\frac{a}{2})\delta_m-1)\right\}
+\sum_{m = 0}^{\lceil {q/c+ad}\rceil -1}md}.$$

Hence, by Remark~\ref{r1} and part~$(1)$ of Lemma~\ref{l8}, and noting that 
$$ md-\frac{q-mc}{a} =  -(\frac{q-m(c+ad)}{a}),$$ 

$$ \begin{array}{lcl}
II(b) & = & \displaystyle{q^2\left(\frac{d}{2(ad+c)}\right) + 
 q\left(\frac{ad}{ad+c}[\frac{1}{a_1c}M_1
-\frac{1}{c}]  \right)}\\
 & & \displaystyle{ + q^0
\left(
(\frac{\epsilon_1 - \epsilon_0 + \Delta_1- \Delta_0}{a_1})M_1 
+ \frac{a}{2}(\sum_{i=0}^{\Delta_0}\delta^2_{i} -
\sum_{i=0}^{\Delta_1}\delta^2_{i})\right.}\\
&& \left. \displaystyle{+(1-\frac{a}{2})(\sum_{i=0}^{\Delta_0}\delta_{i} -
\sum_{i=0}^{\Delta_1}\delta_{i} )+
\epsilon_0-\epsilon_1-\frac{c\epsilon_1}{2a}(\epsilon_1-1)+
\frac{c\epsilon_0}{2a}(\epsilon_0-1) + \frac{d\epsilon_0}{2}(\epsilon_0-1)
}\right).\end{array}$$

Then $$\begin{array}{l}
\displaystyle{ (d+1)(c+\frac{ad}{2})~II = (d+1)(c+\frac{ad}{2})~(II(a)-II(b))}\\
 = \displaystyle{(d+1)(c+\frac{ad}{2})\left\{
q^3\left(\frac{d}{6c(ad+c)}\right) + q^2\left(\frac{1}{4c}+\frac{1}{4(ad+c)}-
\frac{d}{2(ad+c)} \right)\right.}\\
 + \displaystyle{q\left(\frac{1}{2}-\frac{d}{12}-\frac{1}{c}-
\frac{ad}{ad+c}[\frac{1}{a_1c}M_1
-\frac{1}{c}]   \right)}\\
-\displaystyle{q^0\left(\frac{(ad+c)\epsilon_0}{2}\left[\frac{(\epsilon_0-1)}{2} + 
\frac{ad+c}{3a}(\epsilon_0^2-\frac{3}{2}\epsilon_0 +\frac{1}{2})\right]+
\frac{c\epsilon_1}{2}\left[\frac{(\epsilon_1-1)}{2} 
- \frac{c}{3a}(\epsilon_1^2-\frac{3}{2}\epsilon_1 +\frac{1}{2})\right]\right.}\\
+\displaystyle{ 
(\frac{\epsilon_1 - \epsilon_0 + \Delta_1- \Delta_0}{a_1})M_1 
+ \frac{a}{2}(\sum_{i=0}^{\Delta_0}\delta^2_{i} -
\sum_{i=0}^{\Delta_1}\delta^2_{i})}\\
 \left. \displaystyle{+(1-\frac{a}{2})(\sum_{i=0}^{\Delta_0}\delta_{i} -
\sum_{i=0}^{\Delta_1}\delta_{i} )+
\epsilon_0
-\frac{c\epsilon_1}{2a}(\epsilon_1-1) + \frac{c\epsilon_0}{2a}(\epsilon_0-1)
+ \frac{d\epsilon_0}{2}(\epsilon_0-1)}
\right\}.\end{array}$$
Since $$ A = \displaystyle{\sum_{m = 0}^{\lceil q/d\rceil -1}\left\{ d(q-md-1)(q+m)(c+\frac{ad}{2}) + d(q-md-1) \right\}}
+ (d+1)(c+\frac{ad}{2})~II
,$$
 using the above computations we get expression for $A$.\end{proof}

\noindent{\underline{ Proof of the theorem}}: \\
 Recall 
$$HK(X, \sL)(q) =  \sum_{m=0}^{q-1}h^0(X, \sL^m) + \sum_{m\geq 0}|\coker~\Phi_m|,$$ where
we have computed $\sum_{m=0}^{q-1}h^0(X, \sL^m)$ in Equation~({\ref{h1}}).
We can also rewrite 
$$ \sum_{m=0}^{q-1}h^0(X, \sL^m) = q^3\left(\frac{d}{3}(c+\frac{ad}{2})\right)
+ q^2\left((-\frac{d}{2}+\frac{1}{2})(c+\frac{ad}{2})+\frac{d}{2}\right)
+q\left((\frac{d}{6}-\frac{1}{2})(c+\frac{ad}{2})-\frac{d}{2}+1\right).$$
If $c\geq d$ then  applying Lemma~\ref{l2} and Lemma~\ref{l3} to 
Corollary~\ref{l1}, we get 
$\sum_{m\geq 0}|\coker~\Phi_m| = A $. Now, by Lemma~\ref{l6} and the expression given above we get 
$HK(X, \sL)(q)$ as stated in the theorem. 

\vspace{5pt}

If  $c < d$, then  applying Lemma~\ref{l2}, Lemma~\ref{l3} and 
 and Lemma~\ref{l7} to Corollary~\ref{l1}, we get
$$\sum_{m\geq 0}|\coker~\Phi_m|  =  A + B,$$
where
$$\begin{array}{lcl}
B  & = & -
\displaystyle{\left(dc+(d+1)\frac{ad}{2}\right)\left\{\displaystyle{ 
\sum_{m=\lceil q/d\rceil }^{\lceil 
\frac{(a+1)q}{(c+ad)}\rceil-1} 
(md-q+1)\left(q-mc-1-\frac{a}{2}(md-q)\right)}\right.}\\
&  & +\left.\displaystyle{\sum_{m=\lceil{(a+1)q}/{(c+ad)}\rceil}^{\lceil q/c\rceil-1}
\lceil\frac{q-mc}{a}\rceil\left(q-mc-1-\frac{a}{2}
\left(\lceil \frac{q-mc}{a}\rceil -1\right)\right)} \right\}.\end{array}$$
 
Note that 
$$\begin{array}{lcl}
(md-q+1)\left(q-mc-1-\frac{a}{2}(md-q)\right) & = & 
\displaystyle{\frac{(q-mc)^2}{2a} -\frac{((a+1)q-(c+ad)m)^2}{2a}}\\
& & +\displaystyle{ 2q-mc-md-1-\frac{a}{2}(md-q)},
\end{array}$$
where 
$$\sum_{m = \lceil q/d \rceil}^{\lceil {(a+1)q/c+ad}\rceil-1}
\frac{((a+1)q-amd-mc)^2}{2a} = \sum_{m = 0}^{\lceil {(a+1)q/c+ad}\rceil -1}
\frac{((a+1)q-amd-mc)^2}{2a} $$
$$ - \sum_{m = 0}^{\lceil q/d\rceil -1}\left\{\frac{(q-mc)^2}{2a} + 
\frac{a}{2}(q-md)^2 + (q-md)(q-mc)\right\}.$$

Therefore 
$$\sum_{m=\lceil q/d\rceil }^{\lceil \frac{(a+1)q}{(c+ad)}\rceil-1} 
(md-q+1)\left(q-mc-1-\frac{a}{2}(md-q)\right) = -
\sum_{m = 0}^{\lceil {(a+1)q/c+ad}\rceil -1}\frac{((a+1)q-amd-mc)^2}{2a} 
$$
$$+ \sum_{m = 0}^{\lceil q/d\rceil -1}\left\{\frac{(q-mc)^2}{2a} + 
\frac{a}{2}(q-md)^2 + (q-md)(q-mc)\right\} $$
$$ +  \sum_{m=\lceil q/d\rceil }^{\lceil \frac{(a+1)q}{(c+ad)}\rceil-1}
 \left\{\frac{(q-mc)^2}{2a} 
+ 2q-mc-md-1-\frac{a}{2}(md-q)\right\}.$$
On the other hand
 
$$\sum_{m=\lceil{(a+1)q}/{(c+ad)}\rceil}^{\lceil q/c\rceil-1}
\lceil\frac{q-mc}{a}\rceil \left(q-mc-1-\frac{a}{2}
\left(\lceil \frac{q-mc}{a}\rceil -1\right)\right) = $$ 
$$\sum_{m=\lceil{(a+1)q}/{(c+ad)}\rceil}^{\lceil q/c\rceil-1} 
\frac{(q-mc)^2}{2a}-\frac{(q-mc)}{a} + \frac{q-mc}{2} +
 (\frac{a}{2}-1)\delta_m-\frac{a}{2}\delta_m^2.$$
Combining the terms we get

$$\begin{array}{lcl}
B  & = &  
\left( d(c+\frac{(d+1)a}{2})\right) 
 \left\{\displaystyle{ \sum_{m = 0}^{\lceil {(a+1)q/c+ad}\rceil -1}
\frac{((a+1)q-amd-mc)^2}{2a}}\right.\\

& & \displaystyle{- \sum_{m = 0}^{\lceil q/c\rceil -1}\frac{(q-mc)^2}{2a} -
\sum_{m = 0}^{\lceil q/d\rceil -1}\frac{a}{2}(q-md)^2
-\sum_{m = 0}^{\lceil q/d\rceil -1}(q-md)(q-mc)}\\
& & \displaystyle{-\sum_{m = \lceil {(a+1)q/ad+c}\rceil }^{\lceil 
q/c\rceil -1}\left\{(q-mc)\frac{a-2}{2a}+
(\frac{a}{2}-1)\delta_m-\frac{a}{2}\delta_m^2).\right\}}\\
& & \displaystyle{\left.
 -\sum_{m = \lceil q/d\rceil}^{\lceil 
{(a+1)q/ad+c}\rceil -1}(2q-mc-md-1-\frac{a}{2}(md-q))\right\}.}
\end{array}$$

Note that 
$$\sum_{m = \lceil {(a+1)q/ad+c}\rceil }^{\lceil 
q/c\rceil -1} \frac{(a-2)}{2a}(q-mc) + 
 \sum_{m = \lceil q/d\rceil}^{\lceil 
{(a+1)q/ad+c}\rceil -1}(2q-mc-md-1-\frac{a}{2}(md-q))$$
$$ = \sum_{m = 0}^{\lceil 
q/c\rceil -1}\frac{(a-2)}{2a}(q-mc) + 
\sum_{m = 0}^{\lceil {(a+1)q/c+ad}\rceil -1}
\frac{a+2}{2a}((a+1)q-amd-mc)$$
$$
-\sum_{m = 0}^{\lceil q/d\rceil -1}\left\{\frac{a+2}{2}(q-md)+ (q-mc)\right\} 
-\sum_{m = \lceil q/d\rceil}^{\lceil 
{(a+1)q/ad+c}\rceil -1}1.$$

Therefore 
$$\begin{array}{l}
 B  =
  \left( d(c+\frac{(d+1)a}{2})\right) 
 \left\{\displaystyle{ \sum_{m = 0}^{\lceil {(a+1)q/c+ad}\rceil -1}
\frac{((a+1)q-amd-mc)^2}{2a} - \frac{(a+2)}{2a}}((a+1)q-amd-mc)\right.\\
\displaystyle{- \sum_{m = 0}^{\lceil q/c\rceil -1}\left\{\frac{(q-mc)^2}{2a}
+\frac{(a-2)}{2a}(q-mc)\right\}
- \sum_{m = 0}^{\lceil q/d\rceil -1}\left\{\frac{a}{2}(q-md)^2
-\frac{(a+2)}{2}(q-md)\right\}}\\
-\sum_{m = 0}^{\lceil q/d\rceil -1}(q-md-1)(q-mc) -

 \displaystyle{ \left.\sum_{m = \lceil (a+1)q/(ad+c)\rceil}^{\lceil 
q/c\rceil -1}\left((\frac{a}{2}-1)\delta_m-\frac{a}{2}\delta_m^2\right)  
+ \sum_{m = \lceil q/d\rceil}^{\lceil 
(a+1)q/ad+c\rceil-1}1\right\}}.
\end{array} $$

Now, by Remark~\ref{r1},  we can write down the remaining terms of the expression  as follows:

$$\displaystyle{ \sum_{m = 0}^{\lceil \frac{(a+1)q}{c+ad}\rceil -1}
((a+1)q-amd-mc)^2} = \displaystyle{ \sum_{m = 0}^{\lceil
 \frac{(a+1)q}{c+ad}\rceil -1}
((a+1)q-m(c+ad))^2}$$

$$\displaystyle{q^3\left(\frac{(a+1)^3}{3(ad+c)}\right) + q^2\left(\frac{(a+1)^2}{2}\right) 
+ q\left(\frac{(a+1)(ad+c)}{6}\right) + q^0\left(\frac{(ad+c)^2{\bar{\epsilon_0}}}{3}({\bar{\epsilon_0}}^2-
\frac{3}{2}{\bar{\epsilon_0}} +\frac{1}{2})\right)  }$$
and
$$\displaystyle{\sum_{m = 0}^{\lceil q/d\rceil -1}(q-md-1)(q-mc) = 
\frac{1}{d^2}\left[q^3\left(\frac{d}{2}-\frac{c}{6}\right)
+ q^2\left(\frac{d^2}{2} -d +\frac{c}{2}\right)\right.}
\quad\quad\quad\quad\quad\quad\quad\quad\quad\quad\quad\quad $$
$$
 + q\left(\frac{d^3\epsilon_2}{2}(1-\epsilon_2)-
\frac{cd^2\epsilon_2}{2}(1-\epsilon_2)+
\frac{cd^2}{6}+\frac{dc(\epsilon_2-1)}{2}-d^2\epsilon_2 \right)$$
$$ + \left. q^0\left(\frac{cd^3\epsilon_2}{3}(\epsilon_2^2-
\frac{3}{2}\epsilon_2+\frac{1}{2})-\frac{cd^2\epsilon_2(\epsilon_2-1)}{2}
\right)\right].$$  

We see that  coefficient of
 $q^2$ in the term $B$  
$$ = d(c+\frac{(d+1)a}{2})\left[
\left(\frac{(a+1)^2}{4a}-\frac{(a+2)(a+1)^2}{4a(c
+ad)}\right)-\left(\frac{1}{4a}+\frac{a-2}{4ac}\right)-\left(\frac{a}{4}-
\frac{(a+2)}{4d}\right)\right.$$
$$\left.-\left(\frac{1}{2}-\frac{1}{d}+
\frac{c}{2d^2}\right)\right]  = -
d(c+\frac{(d+1)a}{2})\left[\frac{(a+2)(a+1)^2}{4a(c
+ad)} + \frac{a-2}{4ac} -
\frac{(a+2)}{4d} - \frac{1}{d}+
\frac{c}{2d^2}\right]   $$

Now putting together we get $HK(X, \sL)(q)$, for the case $c <  d$.
 
In particular the  
coefficient of $q$ will also involve $\epsilon_2$, $a_1$ 
and $\delta_0$. On the other hand 
coefficient of $q^0$ has terms involving $\epsilon_0$, ${\bar{\epsilon_0}}, \epsilon_2$,
$\epsilon_1$, $\delta_0, \cdots, 
\delta_{a_1-1}$, $\Delta_0$, $\Delta_2$.

\end{document}